\documentclass[12pt]{amsart}
\usepackage{enumerate,amssymb,version,aliascnt,geometry}
\usepackage[all]{xy}
\usepackage{hyperref}
\hypersetup{
pdfauthor={Olivier Haution},
pdftitle={On rational fixed points of finite group actions on the affine space},
hidelinks
}
\DeclareMathOperator{\id}{id}
\DeclareMathOperator{\Spec}{Spec}
\DeclareMathOperator{\rank}{rank}
\DeclareMathOperator{\CH}{CH}
\DeclareMathOperator{\carac}{char}
\DeclareMathOperator{\Tr}{Tr}
\DeclareMathOperator{\Gal}{Gal}
\DeclareMathOperator{\Sing}{Sing}

\newcommand{\inv}{^G}
\newcommand{\cl}[1]{{{#1}_{s}}}
\newcommand{\Oc}{\mathcal{O}}
\newcommand{\Zz}{\mathbb{Z}}
\newcommand{\Qq}{\mathbb{Q}}
\newcommand{\Pp}{\mathbb{P}}
\newcommand{\Zl}{\mathbb{Z}_l}
\newcommand{\Aa}{\mathbb{A}}
\newcommand{\Ql}{{\mathbb{Q}_l}}
\newcommand{\Tan}{T}
\newcommand{\K}{K}
\newcommand{\stacks}[1]{\cite[\href{http://stacks.math.columbia.edu/tag/#1}{Tag {#1}}]{stacks}}

\newtheorem{theorem}{Theorem}[section]

\newaliascnt{proposition}{theorem}
\newtheorem{proposition}[proposition]{Proposition}
\aliascntresetthe{proposition}

\newaliascnt{lemma}{theorem}
\newtheorem{lemma}[lemma]{Lemma}
\aliascntresetthe{lemma}

\newaliascnt{corollary}{theorem}
\newtheorem{corollary}[corollary]{Corollary}
\aliascntresetthe{corollary}

\theoremstyle{definition}

\newaliascnt{remark}{theorem}
\newtheorem{remark}[remark]{Remark}
\aliascntresetthe{remark}

\newaliascnt{example}{theorem}

\aliascntresetthe{example}

\newaliascnt{definition}{theorem}
\newtheorem{definition}[definition]{Definition}
\aliascntresetthe{definition}

\newaliascnt{notation}{theorem}

\aliascntresetthe{notation}

\begin{document}
\begin{abstract}
Consider a finite $l$-group acting on the affine space of dimension $n$ over a field $k$, whose characteristic differs from $l$. We prove the existence of a fixed point, rational over $k$, in the following cases:

--- The field $k$ is $p$-special for some prime $p$ different from its characteristic.

--- The field $k$ is perfect and fertile, and $n = 3$.
\end{abstract}
\author{Olivier Haution}
\title[fixed points of finite group actions on the affine space]{On rational fixed points of finite group actions on the affine space}
\email{olivier.haution at gmail.com}
\address{Mathematisches Institut, Ludwig-Maximilians-Universit\"at M\"unchen, Theresienstr.\ 39, D-80333 M\"unchen, Germany}

\subjclass[2010]{}

\date{\today}

\maketitle

\section{Introduction}
The following question was popularised by Serre in \cite{Serre-finite_fields}.\\

\emph{Let $k$ be a field, and $l$ a prime number different from its characteristic. Let $G$ be a finite $l$-group acting by $k$-automorphisms on the affine space $\Aa^n$. Does the action fix a $k$-rational point?}\\

The structure of the automorphism group of the affine space $\Aa^n$ over a field (the affine Cremona group) is quite mysterious when $n \geq 3$, see e.g.\ \cite{Kraft-Bourbaki}. The question above may be seen as an attempt to provide a small piece of information on this group. Of course the techniques developed to answer this question are likely to be useful to find fixed points for actions on other varieties, but the variety $\Aa^n$ is a particularly interesting test case. The fact that it is not projective prevents the use of many intersection-theoretic techniques. At the same time, one can still hope to answer the question by exploiting the simplicity of $\Aa^n$ from a cohomological point of view (its ``acyclicity'').\\

Serre mentions that the answer to his question is unknown when $G=\Zz/2$, $k=\Qq$ and $n=3$, but gives a positive answer in the two cases listed below.
\begin{itemize}
\item $k$ is finite.
\item $k$ is algebraically closed.
\end{itemize}
He also proves the existence of a fixed point when $k$ is algebraically closed and $G$ is instead cyclic of order prime to the characteristic exponent of $k$.

Esnault and Nicaise extend in \cite{Esnault-Nicaise} the list of fields $k$ for which the question has a positive answer as follows.
\begin{itemize}
\item $k$ is separably closed.
\item $k=\mathbb{R}$.
\item $k$ is a henselian discretely valued field of characteristic zero with algebraically closed residue field of characteristic unequal to $l$.
\item $k$ is a henselian discretely valued field of characteristic zero with residue field finite of cardinality $q$, and such that $l-1 \mid q$.
\end{itemize}
They also settle the question when
\begin{itemize}
\item $k$ is arbitrary and $n\leq 2$.
\end{itemize}
In fact, they prove much more in this case: if $G$ is a solvable group of order prime to $\carac k$, which acts on $\Aa^n$ with $n\leq 2$, then $(\Aa^n)\inv \in \{\Aa^0,\Aa^1,\Aa^2\}$.\\

In the present paper we further extend the list to the following cases.
\begin{itemize}
\item $k$ is pseudo algebraically closed.
\item $k$ is $q$-special for a prime number $q \neq \carac k$.
\item $k$ is perfect and fertile, and $n = 3$.
\end{itemize}

The definitions of ``$q$-special'' and ``fertile'' will be given below; a field is called \emph{pseudo algebraically closed} if every geometrically integral variety has a rational point. In this case, the existence of fixed points follows at once from Smith's theory (see \autoref{prop:Smith}), and will not be further discussed.\\

In fact, we answer the analog of Serre's question where the existence of a rational point is replaced by that of a zero-cycle of degree one (at least over fields of characteristic zero). More precisely, we prove:
\begin{theorem}
\label{th:1}
Let $k$ be a field of characteristic exponent $p$. Let $G$ be a finite group acting on $\Aa^n$. Assume that one of the following assumptions holds.
\begin{enumerate}[(i)]
\item $G$ is cyclic of order prime to $p$.
\item $G$ is an $l$-group, with $l \neq p$.
\end{enumerate}
Then the fixed locus $(\Aa^n)\inv$ supports a zero-cycle whose degree is a power of $p$. If $k$ is perfect and $n \leq 4$, then $(\Aa^n)\inv$ supports a zero-cycle of degree one.
\end{theorem}
The proof is based on Hironaka's resolution of singularities in characteristic zero (and recent results of Cossart-Piltant for threefolds in arbitrary characteristic), on Gabber's theorem on alterations, and on results of K.\ Kato and T.\ Saito concerning wild ramification.

Let $q$ be a prime number. A field $k$ is \emph{$q$-special} if the degree of every finite extension of $k$ is a power of $q$. Over such a field, the presence of a zero-cycle of degree prime to $q$ on a variety is equivalent to that of a rational point. Thus we obtain a positive answer to the original question when $k$ is a $q$-special field with $q\neq \carac k$.\\

A related question concerns the existence of non-trivial separable forms of the affine space: if $X$ is $k$-variety such that $X_L \simeq \Aa^n_L$ for some algebraic separable field extension $L/k$, do we have $X  \simeq \Aa^n_k$? It is easy to see that the answer is positive when $n \in \{0,1\}$. This is also the case when $n=2$ by a result of Shafarevich (see \cite{Sha-infinite}, \cite[(5.8.1)]{GD-Aut-affine}, \cite{Kam-forms}), but again, very little is known as soon as $n \geq 3$, even for $k=\mathbb{R}$ and $L=\mathbb{C}$ (see \cite[Remark 4]{Kraft-Bourbaki}). The arguments of the proof of \autoref{th:1} allow us to show (in \S\ref{sect:forms}) that $X$ at least supports a zero-cycle whose degree is a power of the characteristic exponent of $k$.\\

The second result of this paper concerns the case $n=3$ in the question of Serre:
\begin{theorem}
\label{th:2}
Let $G$ be a finite $l$-group acting on $\Aa^3$ over a perfect field $k$, whose characteristic differs from $l$. Then the action fixes a $k(\!(t)\!)$-rational point.
\end{theorem}
The conclusion of \autoref{th:2} may be reformulated by saying that any compactification of $(\Aa^3)\inv$ has a $k$-rational point. A field is called \emph{fertile} if any dense open subvariety of a smooth variety with a rational point has itself a rational point. This notion was introduced by Pop in \cite{Pop-large} (he calls such fields ``large fields''); the terminology ``fertile'' is due to Moret-Bailly. Any finite extension of a fertile field is fertile, and every henselian field is fertile (more examples and references may be found e.g.\ in \cite{Pop-Survey}). Thus we obtain a positive answer to the question of Serre when $k$ is perfect and fertile, and $n=3$.

The notation and conventions used in the paper are given in \S\ref{sect:Notations}. \autoref{th:1} is proved in \S\ref{sect:zero-cycles}, and \autoref{th:2} in \S\ref{sect:A3} and \S\ref{sect:h1}.\\

\textbf{Acknowledgement:} I thank Jean-Louis Colliot-Th\'el\`ene for his remarks. I thank Mathieu Florence and Philippe Gille for pointing out the applications to separable forms of the affine space. I thank Johannes Nicaise whose suggestions led to considerable simplifications in the proof of \autoref{th:2}.
 
\section{Notation}
\numberwithin{theorem}{subsection}
\numberwithin{lemma}{subsection}
\numberwithin{proposition}{subsection}
\numberwithin{corollary}{subsection}
\numberwithin{example}{subsection}
\numberwithin{notation}{subsection}
\numberwithin{definition}{subsection}
\numberwithin{remark}{subsection}

\label{sect:Notations}
\emph{In the whole paper, we work over a base field $k$, and fix a prime number $l$ unequal to its characteristic.}

\subsection{Varieties} We denote by $\cl{k}$ a separable algebraic closure of $k$. A variety, or $k$-variety, is a reduced separated scheme of finite type over $k$. Closed subsets of a variety will be considered as closed subvarieties using the reduced structure. If $L/k$ is a field extension and $X$ a $k$-variety, we will write $X_L$ for the $L$-variety obtained by extending scalars and taking the underlying reduced scheme. We will say that a variety $X$ is geometrically irreducible, resp.\ connected, if $X_\cl{k}$ is irreducible, resp.\ connected (the empty set is neither connected nor irreducible). The residue field of a point $x$ of $X$ will be denoted by $\kappa(x)$. A variety is complete if it is proper over $k$. A variety is a compactification of $X$ if it is complete and contains $X$ as a dense open subscheme; a compactification always exists by a theorem of Nagata \cite{L-Nagata}.

The index $n_X$ of a variety $X$ is the g.c.d.\ of the degrees of its closed points (it is zero when $X= \varnothing$). 

The subset $\Sing(X)$ of points $x \in X$ such that the local ring $\Oc_{X,x}$ is not regular is closed in $X$ by a theorem of Zariski \cite[(6.12.5)]{ega-4-2}. The variety $X$ is regular when $\Sing(X) = \varnothing$. 

When $X$ is a smooth variety, we denote by $\Tan_X$ its tangent bundle.

We denote by $\Aa^n$ the $k$-variety $\Spec k[t_1,\cdots,t_n]$.

\subsection{Normal crossing divisors}
\label{def:snc}
An effective Cartier divisor $D \to X$ will be called a \emph{snc divisor} if for every $x \in X$ there is a regular system of parameters $f_1,\cdots,f_n$ in $\Oc_{X,x}$ such that the ideal $\Oc(-D)_x$ of $\Oc_{X,x}$ is generated by the element $(f_1)^{r_1}\cdots(f_n)^{r_n}$ for some integers $r_i \geq 0$ (see \cite[Definition~9.1.6]{Liu}). We will say that a closed subvariety of $X$ is the \emph{support of a snc divisor in $X$} if it underlies some snc divisor.

If $Y$ is the support of a snc divisor in $X$, then $X$ is regular. One also sees easily that each irreducible component of $Y$ is regular, and that no point of codimension $n$ in $X$ belongs to more than $n$ irreducible components of $Y$.

\begin{lemma}
\label{lemm:snc_separable}
Let $X$ be a $k$-variety, and $L/k$ a separable algebraic field extension. If $Y$ is the support of a snc divisor in $X$, then so is $Y_L$ in $X_L$.
\end{lemma}
\begin{proof}
It will suffice to prove that if $D \to X$ is a snc divisor, then so is $D_L \to X_L$. Let $y\in X_L$ and $x \in X$ its image. Write $A=\Oc_{X,x}$ and $B=\Oc_{X_L,y}$. Let $f_1,\cdots,f_n$ be a regular system of parameters in $A$ such that the ideal $\Oc_X(-D)_x$ of $A$ is generated by the element $(f_1)^{r_1}\cdots(f_n)^{r_n}$. Then the ideal $\Oc_{X_L}(-D_L)_y$ of $B$ is generated by the image of this element. It will therefore suffice to prove that (the images of) $f_1,\cdots,f_n$ form a regular system of parameters in $B$. Let $C=A \otimes_k L$. The noetherian ring $\kappa(x)\otimes_k L=C/(f_1,\cdots,f_n)$ is integral over $\kappa(x)$ (because $L/k$ is algebraic) and reduced (because $L/k$ is separable). Therefore the ring $C/(f_1,\cdots,f_n)$ is reduced and artinian, and so is its localisation $B/(f_1,\cdots,f_n)$. The latter is additionally local, hence must be a field, showing that $f_1,\cdots,f_n$ generate the maximal ideal of $B$. This concludes the proof, since $\dim B = \dim A =n$ by \cite[(6.1.3)]{ega-4-2} applied to the morphism $A \to B$.
\end{proof}

\subsection{Group actions}
\label{sect:group_actions}
When $G$ is a finite group, an action of $G$ on a variety $X$ will mean an action by $k$-automorphisms. The fixed locus $X\inv$ is a closed subvariety of $X$ such that $X^G(L) = X(L)^G$ for any field extension $L/k$. In fact, there is a scheme-theoretic version of the fixed locus, defined as the intersection in $X$ of the equalisers of $\id_X$ and the action of $\sigma$, where $\sigma$ runs over $G$. The variety $X^G$ is the underlying reduced scheme. When $X$ is smooth and the order of $G$ is prime to the characteristic of $k$, then the scheme-theoretic fixed locus is smooth \cite[Proposition 3.4]{Edixhoven-Neron}, hence coincides with the variety $X^G$.

\subsection{Cohomology groups} Let $X$ be a variety. We will write $H^i(X_\cl{k},\Ql)$ (resp.\  $H^i_c(X_\cl{k},\Ql)$) for the \'etale cohomology groups (resp.\ with compact supports) with $\Ql$-coefficients of the $\cl{k}$-variety $X_\cl{k}$. 

\subsection{Graphs}
By a graph, we will mean be a finite undirected graph. In other words, a graph $\Gamma$ consists in a finite set of vertices $V(\Gamma)$, and for each unordered pair of vertices $\{v_1,v_2\}$ a finite set of edges $E(\{v_1,v_2\})$. If $e \in E(\{v_1,v_2\})$ we say that $v_1$ and $v_2$ are the extremities of the edge $e$, or that $e$ is an edge between $v_1$ and $v_2$. Declaring two vertices equivalent if there is an edge between them generates an equivalence relation on $V(\Gamma)$. The graph $\Gamma$ will be called connected if there is exactly one equivalence class. A tree is a connected graph with $n-1$ edges and $n$ vertices, for some $n \geq 1$. A vertex $v$ of a tree is called a leaf if there is at most one edge one of whose extremities is $v$.

\section{Zero-cycles of degree one}
\label{sect:zero-cycles}

\subsection{Euler characteristic and index}
\begin{definition}
The Euler characteristic (with compact supports) of a variety $X$ is defined as:
\[
\chi(X) = \sum_i (-1)^i \dim_{\Ql} H^i_c(X_\cl{k},\Ql).
\]
\end{definition}
If $Z$ is a closed subvariety of $X$, then the long exact sequence of cohomology groups with compact supports \cite[III, Remark 1.30]{Milne-etale} yields the relation
\[
\chi(X) = \chi(Z) + \chi(X-Z).
\]

In the next proposition, we denote by $c$ the total Chern class with values in the Chow group.
\begin{proposition}
\label{prop:degctan}
Let $X$ be a smooth complete variety. Then $\chi(X) = \deg c(\Tan_X)$. In particular $X$ supports a zero-cycle of degree $\chi(X)$.
\end{proposition}
\begin{proof}
This follows from the Lefschetz trace formula \cite[VI, Theorem~12.3]{Milne-etale} and the self-intersection formula \cite[Example~8.1.12]{Ful-In-98}.
\end{proof}

\begin{lemma}
\label{lemm:perfect_closure}
Let $L/k$ be a purely inseparable field extension, and $X$ a non-empty $k$-variety. Then $n_X / n_{X_L}$ is a power of the characteristic exponent of $k$.
\end{lemma}
\begin{proof}
Let $p$ be the characteristic exponent of $k$. Let $x$ be a closed point of $X_L$. Its image $y$ in $X$ is a closed point, because the extension $L/k$ is algebraic. By multiplicativity of separable degrees \cite[V, Theorem~4.1]{Lang-Algebra}, we have in $\mathbb{N} \cup \{\infty\}$
\[
[\kappa(x) :L]_s \cdot [L:k]_s = [\kappa(x) : k]_s = [\kappa(x) : \kappa(y)]_s \cdot [\kappa(y) :k]_s.
\]
By assumption we have $[L:k]_s=1$, and since the extension $\kappa(x)/L$ is finite, all the above displayed separable degrees are finite. It follows that
\[
[\kappa(y) :k]_s \mid [\kappa(x) :L]_s.
\]
Since the extension $\kappa(y)/k$ is finite, there is an integer $m$ such that $[\kappa(y) :k]= p^m \cdot [\kappa(y) :k]_s$. Thus we have a chain of divisibilities
\[
n_X \mid [\kappa(y) :k] \mid  p^m \cdot [\kappa(x) :L]_s \mid p^m\cdot [\kappa(x) :L].
\]

Let now $x_1,\cdots,x_n \in X_L$ be a finite family of closed points such that $n_{X_L}$ is the g.c.d.\ of $[\kappa(x_1):L], \cdots, [\kappa(x_n):L]$. As we have just seen, we may find for each $i$ an integer $m_i$ such that $n_X \mid p^{m_i}\cdot [\kappa(x_i):L]$. Letting $m$ be the maximum of the $m_i$'s, we obtain that $n_X \mid p^m \cdot n_{X_L}$. 

On the other hand any zero-cycle on $X$ gives rise by scalars extension to a zero-cycle of the same degree on the $L$-scheme $X \otimes_k L$, and therefore also on the underlying $L$-variety $X_L$. Thus $n_{X_L} \mid n_X$, and the statement follows.
\end{proof}

\begin{proposition}
\label{prop:chi_index}
Let $X$ be a $k$-variety.
\begin{enumerate}[(i)]
\item \label{prop:chi:1} If $\carac k=0$, or if $k$ is perfect and $\dim X \leq 3$, then $X$ supports a zero-cycle of degree $\chi(X)$.
\item \label{prop:chi:2} If $\carac k=p>0$, then $X$ supports a zero-cycle of degree $p^m\cdot\chi(X)$, for some integer $m$.
\end{enumerate}
\end{proposition}
\begin{proof}
\eqref{prop:chi:1}: We proceed by induction on $\dim X$; if $X=\varnothing$, then $\chi(X)=0$ and $X$ supports the null zero-cycle. While proving the statement, one may replace $X$ by any dense open subvariety. Indeed assume that such a subvariety $U$ of $X$ supports a zero-cycle of degree $\chi(U)$. Then so does $X$. Since by induction $X-U$ supports a zero-cycle of degree $\chi(X-U)$, so does again $X$. Therefore $X$ supports a zero-cycle of degree $\chi(X)=\chi(U) + \chi(U-X)$. 

In particular we may replace $X$ by a smooth dense open subvariety (which exists since $k$ is perfect). Then we may find a smooth compactification $X'$ of $X$; this result is due to Hironaka \cite{Hir-64} when $\carac k = 0$, to Abhyankar (see e.g.\ Lipman's \cite{Lipman-desingularization}) when $\dim X \leq 2$, and to Cossart and Piltant \cite{CP-3} when $\dim X \leq 3$. Then $X'$ supports a zero-cycle of degree $\chi(X')$ by \autoref{prop:degctan}, and by induction $Z=X'-X$ (hence also $X'$) supports a zero-cycle of degree $\chi(Z)$. It follows that $X'$ supports a zero-cycle of degree $\chi(X') - \chi(Z)=\chi(X)$. The same is true for its dense open subvariety $X$, by a moving lemma \cite[Proposition~6.8]{Index} (see also \cite[p.599]{Colliot-Finitude}).\\

\eqref{prop:chi:2}: We let $q$ be a prime number different from $p$ and denote by $v_q$ the $q$-adic valuation on $\Zz$. We need to prove that $v_q(n_X) \leq v_q(\chi(X))$. Replacing $k$ by a perfect closure affects neither $\chi(X)$ \cite[VIII, Th\'eor\`eme 1.1]{SGA-4-2} nor $v_q(n_X)$ (\autoref{lemm:perfect_closure}). Thus we may assume that $k$ is perfect. We again proceed by induction on $\dim X$. Induction on the number of irreducible components of $X$ shows that we may assume that $X$ is irreducible (the argument is the same as beginning of the proof of \eqref{prop:chi:1}). Let $X'$ be a compactification of $X$. By the results of Gabber on alterations \cite[Introduction, Theorem 3 (1), or X, Theorem 3.5 (iii)]{Gabber-book}, we may find a smooth complete variety $Y'$, a morphism $f\colon Y' \to X'$, a non-empty open subvariety $U$ of $X'$ such that $V=f^{-1}U$ is dense in $Y'$ and $V \to U$ is \'etale and finite of degree $d$, with $d$ prime to $q$. Shrinking $U$, we may assume that $U$ is smooth and contained in $X$. It will suffice to prove that $v_q(n_U) \leq v_q(\chi(U))$ (see the argument at the beginning of the proof of \eqref{prop:chi:1}). Consider the variety $B=X'-U$ (with reduced structure). Let $X''$ be the blow-up of $B$ in $X'$, and $Y''$ the blow-up of the (possibly non-reduced) closed subscheme $f^{-1}B$ of $Y'$. There is an induced morphism $Y'' \to X''$. Using once again the result of Gabber, we find a smooth complete variety $Z''$, a morphism $g\colon Z'' \to Y''$ generically of degree prime to $q$ and such that $W=g^{-1}V$ is the complement of a snc divisor in $Z''$. We now use results of K.\ Kato et T.\ Saito. We have by \cite[Lemma~3.4.5.1]{Kato-Saito}
\begin{equation}
\label{eq:dlog}
\deg D^{log}_{V/U} = d \cdot \chi(U) - \chi(V) \in \Qq,
\end{equation}
where $D^{log}_{V/U} \in \CH_0(\overline{V} \backslash V)\otimes_{\Zz} \Qq$ is the wild different \cite[Definition 3.4.1]{Kato-Saito}, and $\CH_0(\overline{V} \backslash V)$ is the Chow group of zero-cycles on the boundary \cite[Definition 3.1.1]{Kato-Saito}. We apply \cite[Theorem~3.2.3.1]{Kato-Saito}, where the diagram (3.4) of \cite{Kato-Saito} is
\[ \xymatrix{
W\ar[r] \ar[d] & Z'' \ar[d] \ar[ddr]&\\ 
V \ar[r] \ar[d]& Y'' & \\
U \ar[rr]&& X''
}\]
and obtain that the image of $D^{log}_{V/U}$ under the natural morphism $\CH_0(\overline{V} \backslash V)\otimes_{\Zz} \Qq \to \CH_0(Y''-V)\otimes_{\Zz} \Qq$ lies in the image of the morphism $\CH_0(Y''-V) \otimes_{\Zz} \Zz_{(q)} \to \CH_0(Y''-V)\otimes_{\Zz} \Qq$ (here $\Zz_{(q)}\subset \Qq$ denotes those fractions whose denominator is prime to $q$). Since there is a morphism $Y''-V \to Y'$, we deduce that
\begin{equation}
\label{eq:vq_Dlog}
v_q(n_{Y'}) \leq v_q(n_{Y''-V}) \leq v_q(\deg D^{log}_{V/U}).
\end{equation}
On the other hand $v_q(n_{Y'}) \leq v_q(\chi(Y'))$ (by \autoref{prop:degctan}, since $Y'$ is smooth), and by induction $v_q(n_{Y'-V}) \leq v_q(\chi(Y'-V))$. Using the relation $n_{Y'} \mid n_{Y'-V}$, we obtain
\begin{equation}
\label{eq:vq_nY}
v_q(n_{Y'}) \leq v_q(\chi(Y') - \chi(Y'-V)) = v_q(\chi(V)).
\end{equation}
Combining \eqref{eq:dlog}, \eqref{eq:vq_Dlog}, \eqref{eq:vq_nY}, we deduce that $v_q(n_{Y'}) \leq v_q(\chi(U))$. Finally $n_{Y'}=n_V$ (by the moving lemma \cite[Proposition~6.8]{Index}, since $V$ is dense in the smooth variety $Y'$), and $n_U \mid n_V$ (since there is a morphism $V \to U$). We conclude that $v_q(n_{U}) \leq v_q(\chi(U))$, as required.
\end{proof}

\begin{remark}
Statement \eqref{prop:chi:1} could be improved (by removing the assumption $\dim X \leq 3$, for $k$ perfect) if we could resolve singularities in positive characteristic. On the other hand, we may see that \eqref{prop:chi:2} is optimal by considering the case when $X$ is the spectrum of a finite purely inseparable field extension of $k$. 
\end{remark}

\subsection{Actions of \texorpdfstring{$l$}{l}-groups on \texorpdfstring{$\Aa^n$}{An}}

\begin{proposition}
\label{prop:Smith}
Let $l$ be a prime number unequal to $\carac k$. Let $G$ be a finite $l$-group acting on $\Aa^n$ over $k$, and $S=(\Aa^n)\inv$. Then $S$ is smooth, geometrically connected, and satisfies $\chi(S)=1$. More precisely:
\[
H^i(S_\cl{k},\Ql)= \left\{ \begin{array}{ll}
		  0 &\mbox{ if $i \neq 0$} \\
		  \Ql &\mbox{ if $i=0$}
       		\end{array},
	\right. \text{ and } \; H^i_c(S_\cl{k},\Ql)= \left\{ \begin{array}{ll}
		  0 &\mbox{ if $i \neq 2\dim S$} \\
		  \Ql &\mbox{ if $i=2\dim S$}
       		\end{array}.
	\right.
\]
\end{proposition}
\begin{proof}
The variety $S$ is smooth (see \S\ref{sect:group_actions}), and by Poincar\'e duality it will suffice to prove the statements concerning $H^i(S_\cl{k},\Ql)$. Replacing $k$ by a perfect closure do not affect $H^i(S_\cl{k},\Ql)$ by \cite[VIII, Th\'eor\`eme 1.1]{SGA-4-2}, so that we may assume that $k$ is perfect.

Let us say that a variety $X$ is $A$-acyclic if the \'etale cohomology groups with coefficients in $A$ satisfy $H^0_{\text{\'et}}(X_{\cl{k}},A)=A$ and $H^i_{\text{\'et}}(X_{\cl{k}},A)=0$ for $i \neq 0$. The variety $\Aa^n$ is $\Zz/l$-acyclic \cite[Corollary~4.20]{Milne-etale}, and it follows from Smith's theory \cite[Theorem 7.5 and its corollary]{Serre-finite_fields} that $S$ is also $\Zz/l$-acyclic. Using the exact sequence
\[
0 \to \Zz/l^m \to \Zz/l^{m+1} \to \Zz/l \to 0
\]
we see by induction on $m$ that $S$ is $\Zz/l^m$-acyclic for every positive integer $m$. Taking the limit over $m$, it follows that $S$ is $\Zz_l$-acyclic, and tensoring with $\Ql$ over $\Zl$, we conclude that $S$ is $\Ql$-acyclic. 
\end{proof}

\begin{theorem}
\label{th:l-group}
Let $k$ be a field of characteristic exponent $p$, and $l$ a prime number unequal to $p$. Let $G$ be a finite $l$-group acting on $\Aa^n$ over $k$. Then $(\Aa^n)\inv$ supports a zero-cycle whose degree is a power of $p$. If $k$ is perfect and $n \leq 4$, then $(\Aa^n)\inv$ supports a zero-cycle of degree one.
\end{theorem}
\begin{proof}
Since $\chi((\Aa^n)\inv)=1$ by \autoref{prop:Smith}, the statement follows from \autoref{prop:chi_index} (we may assume that $G$ acts non-trivially, so that $\dim (\Aa^n)\inv \leq 3$ when  $n \leq 4$).
\end{proof}

\subsection{Actions of cyclic groups on \texorpdfstring{$\Aa^n$}{An}}

\begin{theorem}
\label{th:cyclic-group}
Let $k$ be a field of characteristic exponent $p$. Let $G$ be a finite cyclic group of order prime to $p$ acting on $\Aa^n$ over $k$. Then $(\Aa^n)\inv$ supports a zero-cycle whose degree is a power of $p$. If $k$ is perfect and $n \leq 4$, then $(\Aa^n)\inv$ supports a zero-cycle of degree one.
\end{theorem}
\begin{proof}
By \autoref{lemm:perfect_closure} we may assume that $k$ is perfect. Let $g$ be a generator of $G$. By a result of Deligne-Lusztig \cite[Theorem~3.2]{Deligne-Lusztig} we have
\begin{equation}
\label{eq:DL}
\sum_i (-1)^i \Tr\big(g^* : H^i_c(\Aa^n_\cl{k},\Ql)\big) = \chi\big((\Aa^n)\inv\big).
\end{equation}
Using \cite[Corollary~4.20]{Milne-etale} for $F=\Zz/l^m$ and taking the limit over $m$ and tensoring with $\Ql$ over $\Zl$, and using Poincar\'e duality, we see that $H^i_c(\Aa^n_\cl{k},\Ql) = 0$ if $i\neq 2n$ and $H^{2n}_c(\Aa^n_\cl{k},\Ql)=\Ql$. Moreover $G$ acts trivially on the set of connected components of $\Aa^n_\cl{k}$, and thus also on $H^0(\Aa^n_\cl{k},\Ql) \simeq H^{2n}_c(\Aa^n_\cl{k},\Ql)$. Therefore the value of the left hand side of \eqref{eq:DL} is $1$, and we conclude using \autoref{prop:chi_index}, as in the proof of \autoref{th:l-group}.
\end{proof}

\subsection{\texorpdfstring{$q$}{q}-special fields}
Let $q$ be a prime number. The field $k$ is called \emph{$q$-special} if the degree of every finite extension of $k$ is a power of $q$. The combination of \autoref{th:l-group} and \autoref{th:cyclic-group} is \autoref{th:1} of the introduction, and has the following consequence.
\begin{corollary}[of \autoref{th:1}]
Let $k$ be a field of characteristic exponent $p$. Assume that $k$ is $q$-special, for some prime number $q$ unequal to $p$. Then the action on $\Aa^n$ of any finite $l$-group with $l \neq p$, or any finite cyclic group of order prime to $p$, fixes a $k$-rational point.
\end{corollary}
This corollary applies to real-closed fields, which  are $2$-special and of characteristic zero. In particular, we obtain a purely algebraic proof for the case $k=\mathbb{R}$ (the classical proof uses algebraic topology, see e.g.\ \cite[\S5.4]{Esnault-Nicaise}).

\subsection{Separable forms of \texorpdfstring{$\Aa^n$}{An}}
\label{sect:forms}
As mentioned in the introduction, it is currently unknown whether every $\mathbb{R}$-variety $X$ such that $X_\mathbb{C} \simeq \Aa^n_\mathbb{C}$ is isomorphic to $\Aa^n_\mathbb{R}$ (for $n\geq 3$). A consequence of the next proposition is that $X$ must have an $\mathbb{R}$-point.

\begin{proposition}
Let $k$ be a field of characteristic exponent $p$, and $X$ a $k$-variety such that $X_{\cl{k}} \simeq \Aa^n_{\cl{k}}$. Then $X$ supports a zero-cycle whose degree is a power of $p$. If $k$ is perfect and $n \leq 3$, then $X$ supports a zero-cycle of degree one. If $k$ is $q$-special for some prime number $q$ unequal to $p$, then $X(k) \neq \varnothing$.
\end{proposition}
\begin{proof}
This follows from \autoref{prop:chi_index} and the fact that $\chi(\Aa^n) =1$, already observed in the course of the proof of \autoref{th:cyclic-group}.
\end{proof}

\section{Actions of \texorpdfstring{$l$}{l}-groups on \texorpdfstring{$\Aa^3$}{A3}}
\numberwithin{theorem}{section}
\numberwithin{lemma}{section}
\numberwithin{proposition}{section}
\numberwithin{corollary}{section}
\numberwithin{example}{section}
\numberwithin{notation}{section}
\numberwithin{definition}{section}
\numberwithin{remark}{section}

\label{sect:A3}
A field $F$ is called \emph{fertile} if any dense open subvariety of a smooth $F$-variety with an $F$-rational point has itself an $F$-rational point. When a (smooth) $k$-variety $S$ admits a smooth compactification $S'$, the following conditions are equivalent.
\begin{itemize}
\item Every compactification of $S$ has a $k$-rational point.

\item The variety $S$ has an $F$-rational point for any fertile field $F$ containing $k$.
\item The variety $S$ has a $k(\!(t)\!)$-rational point.
\end{itemize}

Indeed, any of these conditions is equivalent to the condition $S'(k) \neq \varnothing$ (this follows from Nishimura's Lemma, the valuative criterion of properness and the fact that $k(\!(t)\!)$ is fertile, see \cite[\S1.A.2)]{Pop-Survey}).

\begin{theorem}
\label{th:three}
Let $k$ be a perfect field, and $l$ a prime number different from its characteristic. Let $G$ be a finite $l$-group acting on $\Aa^3$ over $k$. Then the variety $S=(\Aa^3)\inv$ satisfies the three above conditions.
\end{theorem}
\begin{proof}
By \autoref{prop:Smith}, the variety $S$ is geometrically connected, smooth and satisfies $\chi(S)=1$. If $\dim S=0$, then $S$ must be a single rational point. If $\dim S=3$, then $S=\Aa^3$. In these two cases, the conclusion of the theorem holds.

We now assume that $\dim S \in \{1,2\}$. Since $k$ is perfect, the variety $S$ admits a smooth compactification $S'$ such that the closed subvariety $D=S'-S$ is the support of a snc divisor in $X$ (\S\ref{def:snc}), by \cite{Lipman-desingularization} and \cite[Theorem 9.2.26]{Liu}. It will suffice to prove that $S'(k) \neq \varnothing$. The variety $S'$ is geometrically connected because its dense open subvariety $S$ is so. Moreover $H^1_c(S_\cl{k},\Ql)=0$ by \autoref{prop:Smith}. The exact sequence of $\Ql$-vector spaces
\[
H^0(S'_\cl{k},\Ql) \to H^0(D_\cl{k},\Ql) \to H^1_c(S_\cl{k},\Ql)
\]
then shows that $D$ is geometrically connected. 

In case $\dim S=1$, the variety $D$ is non-empty because $S$ is affine. Thus $\dim D=0$, and $D$ must be a single rational point ($k$ is perfect), which concludes the proof in this case. But we can say more: the exact sequence of $\Ql$-vector spaces
\[
H^1_c(S_\cl{k},\Ql) \to H^1(S'_\cl{k},\Ql) \to H^1(D_\cl{k},\Ql)
\]
shows that $H^1(S'_\cl{k},\Ql)=0$. Thus the $k$-variety $S'$ is a smooth complete curve of genus zero (see e.g.\ \cite[IX, \S4]{SGA-4-3}), that is, a conic. It has a rational point, hence is isomorphic to $\Pp^1$, and $S= \Pp^1 -D \simeq \Aa^1$.

We now assume that $\dim S=2$. By semi-purity \cite[VI, Lemma~9.1]{Milne-etale}, the restriction morphisms $H^1_{\text{\'et}}(S'_\cl{k},\Zz/l^m) \to H^1_{\text{\'et}}(S_\cl{k},\Zz/l^m)$ are injective for all $m$, and by left-exactness of the inverse limit and flatness of $\Ql$ over $\Zl$, it follows that $H^1(S'_\cl{k},\Ql) \to H^1(S_\cl{k},\Ql)$ is injective. Since by \autoref{prop:Smith} we have $H^1(S_\cl{k},\Ql)=0$, we deduce that $H^1(S'_\cl{k},\Ql)=0$. \autoref{prop:Smith} also yields $H^2_c(S_\cl{k},\Ql)=0$, and the exact sequence of $\Ql$-vector spaces
\[
H^1(S'_\cl{k},\Ql) \to H^1(D_\cl{k},\Ql) \to H^2_c(S_\cl{k},\Ql)
\]
shows that $H^1(D_\cl{k},\Ql)=0$. By \autoref{lemm:Tan}, the variety $D$ supports a zero-cycle of degree $\chi(S)$, and $\chi(S)=1$ by \autoref{prop:Smith}. We prove in \autoref{prop:h1} below that $D(k) \neq \varnothing$, which implies that $S'(k) \neq \varnothing$. 
\end{proof}

\begin{remark}
Let $k$ be an arbitrary field, and $l$ a prime number different from its characteristic. Let $G$ be a finite $l$-group acting on $\Aa^n$ over $k$, with $n$ arbitrary. Let $S=(\Aa^n)\inv$. The proof of \autoref{th:three} shows the following.
\begin{itemize}
\item If $\dim S = n$, then $S=\Aa^n$.
\item If $\dim S = 0$, then $S$ is a single $k$-rational point
\item If $\dim S = 1$ and $k$ is perfect, then $S \simeq \Aa^1$. In case $n=2$, this was proved in \cite[Theorem~5.12]{Esnault-Nicaise} without assuming that $k$ is perfect.
\end{itemize}
However, our proof breaks down when $\dim S=2$ and $n >3$ because \autoref{lemm:Tan} below seems to be specific to subvarieties of codimension one in $\Aa^n$.
\end{remark}

\begin{lemma}
\label{lemm:Tan}
Assume that the field $k$ is perfect. Let $S$ be a smooth closed subvariety of pure codimension one in $\Aa^n$, with $2 \leq n \leq 4$. Let $S'$ be a smooth compactification of $S$. Then $D=S'-S$ supports a zero-cycle of degree $\chi(S)$.
\end{lemma}
\begin{proof}
The closed embedding $i\colon S \to \Aa^n$ is an effective Cartier divisor; let $L=\Oc(S)$ be the corresponding line bundle on $\Aa^n$. Then we have an exact sequence of vector bundles on $S$ \cite[(17.13.2.1)]{ega-4-4}
\[
0 \to \Tan_S \to i^*\Tan_{\Aa^n} \to i^*L \to 0.
\]
Thus $[\Tan_S] \in K_0(S)$ is in the image of the morphism $i^* \colon K_0(\Aa^n) \to K_0(S)$. By homotopy invariance \cite[\S6, Corollary of Theorem 9]{Qui-72}, the natural morphism $\Zz \to K_0(\Aa^n)$ is an isomorphism, and we deduce that $[\Tan_S]= \rank \Tan_S = n-1 \in K_0(S)$. This implies that $c_{n-1}(\Tan_S) = 0 \in \CH_0(S)$, because the Chern classes of a vector bundle depend only on its class in $\K_0$ \cite[Example~3.2.7 (b)]{Ful-In-98}, and the $(n-1)$-st Chern class of a trivial bundle vanishes (since $n -1 >0$). Now $\Tan_S = (\Tan_{S'})|_S$, hence $c_{n-1}(\Tan_{S'})|_S=0\in \CH_0(S)$. By the localisation sequence for Chow groups \cite[Proposition~1.8]{Ful-In-98}, it follows that the cycle class $c_{n-1}(\Tan_{S'})$ lies in the image of the morphism $\CH_0(D) \to \CH_0(S')$. Thus $D$ supports a zero-cycle of degree $\deg c_{n-1}(\Tan_{S'})$, and $\deg c_{n-1}(\Tan_{S'}) = \chi(S')$ by \autoref{prop:degctan}. Since the variety $D$ also supports a zero-cycle of degree $\chi(D)$ by \autoref{prop:chi_index}~\eqref{prop:chi:1} (because $n-1 \leq 3$), it must support one of degree $\chi(S) = \chi(S') - \chi(D)$.
\end{proof}

\section{One-dimensional snc divisors with no $H^1$}
\numberwithin{theorem}{subsection}
\numberwithin{lemma}{subsection}
\numberwithin{proposition}{subsection}
\numberwithin{corollary}{subsection}
\numberwithin{example}{subsection}
\numberwithin{notation}{subsection}
\numberwithin{definition}{subsection}
\numberwithin{remark}{subsection}

\label{sect:h1}
\subsection{The geometric number of components}

\begin{definition}
\label{def:m_mu}
Let $X$ be a variety. We define $m_X$, resp.\ $\mu_X$, as the number of irreducible, resp.\ connected, components of the $\cl{k}$-variety $X_{\cl{k}}$.
\end{definition}

When $X$ is a variety, we have
\begin{equation}
\label{eq:h0_mu}
\mu_X = \dim_{\Ql} H^0(X_\cl{k},\Ql).
\end{equation}

\begin{lemma}
\label{lemm:h2n_m}
Let $X$ be a variety of pure dimension $n$. Then 
\[
m_X = \dim_{\Ql} H^{2n}_c(X_\cl{k},\Ql).
\]
\end{lemma}
\begin{proof}
Replacing $k$ by a perfect closure affects neither $m_X$ (this operation does not affect the topological space $X_\cl{k}$) nor  $H^{2n}_c(X_\cl{k},\Ql)$ \cite[VIII, Th\'eor\`eme 1.1]{SGA-4-2}. Thus we may assume that $k$ is perfect. Let $U$ be a smooth dense open subvariety of $X$, and $Z$ its complement. We have an exact sequence of $\Ql$-vector spaces
\[
H^{2n-1}_c(Z_\cl{k},\Ql) \to H^{2n}_c(U_\cl{k},\Ql) \to H^{2n}_c(X_\cl{k},\Ql) \to H^{2n}_c(Z_\cl{k},\Ql),
\]
where the two extreme groups vanish since $\dim Z < n$. Using Poincar\'e duality and \eqref{eq:h0_mu}, we deduce that
\[
\dim_{\Ql} H^{2n}_c(X_\cl{k},\Ql) = \dim_{\Ql} H^{2n}_c(U_\cl{k},\Ql) =  \dim_{\Ql} H^0(U_\cl{k},\Ql) = \mu_U.
\]
The result follows, since $m_X=m_U$ (the set $U_{\cl{k}}$ is dense in the noetherian space $X_{\cl{k}}$), and $m_U=\mu_U$ (the scheme $U_{\cl{k}}$ is locally irreducible, being smooth).
\end{proof}

\begin{lemma}
\label{lemm:mu_n}
Let $X$ be a connected complete variety. Then $\mu_X \mid n_X$.
\end{lemma}
\begin{proof}
When $X = \Spec F$ for a finite field extension $F/k$, the integer $\mu_X$ is the separable degree $[F:k]_s$, which divides the degree $[F:k]=n_X$.

In general, the $k$-algebra $F=H^0(X,\Oc_X)$ is reduced and finite, and moreover the $k$-variety $\Spec F$ is connected. Thus $F$ is a finite field extension of $k$. Since $\mu_{\Spec F} =\mu_X$, and $n_{\Spec F} \mid n_X$, we conclude using the special case treated above. 
\end{proof}

\subsection{The dual graph}
\begin{definition}
Let $X$ be a $k$-variety. The \emph{dual graph of $X$}, denoted henceforth $\Gamma_X$, is the undirected graph defined as follows. Its vertices are the irreducible components of of the $\cl{k}$-variety $X_{\cl{k}}$ (its cardinality is thus $m_X$). The set of edges between two irreducible components is empty if they coincide, and equal to the set of irreducible components (over $\cl{k}$) of their intersection otherwise.
\end{definition}

The graph $\Gamma_X$ is naturally endowed with an action of the absolute Galois group $\Gal(\cl{k}/k)$, and orbits in the set of vertices correspond bijectively to irreducible components of $X$ by \stacks{04KY}.

\begin{lemma}
A variety $X$ is geometrically connected if and only if its dual graph $\Gamma_X$ is connected.
\end{lemma}
\begin{proof}
First note that $X= \varnothing$ if and only if $\Gamma_X=\varnothing$. Associating to a set of vertices $V$ the union $\alpha(V)$ of the corresponding irreducible components yields a bijection between the subsets of $V(\Gamma_X)$ and the subsets of $X_\cl{k}$ which are unions of irreducible components of $X$. We have $\alpha(V_1 \cup V_2) = \alpha(V_1) \cup \alpha(V_2)$. Moreover $\alpha(V_1) \cap \alpha(V_2)=\varnothing$ if and only if $V_1 \cap V_2 = \varnothing$ and there are no edges between elements $V_1$ and $V_2$. To conclude the proof, note that a closed and open subset of $X_\cl{k}$ is necessarily a union of irreducible components of $X_\cl{k}$.
\end{proof}

\subsection{Zero-cycles of odd degree and rational points}
\begin{proposition}
\label{prop:h1}
Assume that $k$ is perfect. Let $X$ be a complete, geometrically connected variety of dimension one such that $H^1(X_\cl{k},\Ql)=0$. Assume that $X$ is the support of a snc divisor in some variety (\S\ref{def:snc}). If $X$ supports a zero-cycle of odd degree, then $X(k) \neq \varnothing$.
\end{proposition}
\begin{proof}
Since $\dim X \leq 1$, the variety $I=\Sing(X)$ is finite. It follows from \cite[(6.7.4)]{ega-4-2} that $I_{\cl{k}}=\Sing(X_{\cl{k}})$. Let $X_1,\cdots,X_{m_X}$ be the irreducible components of $X_{\cl{k}}$. Since $X_{\cl{k}}$ is the support of a snc divisor in some variety (\autoref{lemm:snc_separable}), each point of $I_{\cl{k}}$ belongs to exactly two $X_i$'s. Conversely any point of $X_i \cap X_j$ with $i\neq j$ belongs to $I_{\cl{k}}$. It follows that $I_{\cl{k}}$ is the disjoint union of the varieties underlying $X_i \cap X_j$ for $i < j$, and that $\chi(I)$ is the number of edges of $\Gamma_X$. By \autoref{lemm:triple_union} below, it follows that
\[
\chi(X) = \sum_i \chi(X_i) - \sum_{i < j} \chi(X_i \cap X_j) = \sum_i \chi(X_i) - \chi(I)
\]
From the assumption on $H^1(X_{\cl{k}},\Ql)$ and by \autoref{lemm:h2n_m}, we deduce that
\begin{equation}
\label{eq:sum}
1+m_X = \sum_{i=1}^{m_X} \chi(X_i) - \chi(I).
\end{equation}
Since for each $i$, the $\cl{k}$-variety $X_i$ is irreducible, we have 
\begin{equation}
\label{eq:chi2}
\chi(X_i) = 1- \dim_{\Ql} H^1(X_i,\Ql) +1 \leq 2.
\end{equation}
From \eqref{eq:sum} and \eqref{eq:chi2} we deduce that $\chi(I) \leq m_X-1$ with equality if and only if $H^1(X_i,\Ql)=0$ for each $i$. But $\Gamma_X$ is a connected graph with $m_X$ vertices and $\chi(I)$ edges, hence $\chi(I) \geq m_X-1$. Thus $\chi(I)=m_X -1$ and $\Gamma_X$ is a tree. In addition, for each $i$, the $\cl{k}$-variety $X_i$ is a smooth connected curve, whose genus is zero because $H^1(X_i,\Ql)=0$ by \eqref{eq:chi2} (see e.g.\ \cite[IX, \S4]{SGA-4-3}). It follows that $X_i \simeq \Pp^1_{\cl{k}}$ for each $i$. We conclude using \autoref{lemm:tree} below.
\end{proof}

\begin{lemma}
\label{lemm:triple_union}
Let $X_1,\cdots,X_n$ be closed subvarieties of $X$ such that $X_i \cap X_j \cap X_l=\varnothing$ whenever $i,j,l$ are pairwise distinct. If $X=X_1 \cup \cdots \cup X_n$, then
\[
\chi(X) = \sum_i \chi(X_i) - \sum_{i<j} \chi(X_i \cap X_j).
\]
\end{lemma}
\begin{proof}
We proceed by induction on $n$, the case $n=0$ being clear. When $n >0$, let $Y=X_1 \cup \cdots \cup X_{n-1}$. Then $Y\cap X_n$ is the disjoint union of the varieties $X_i\cap X_n$ for $i < n$, and using the induction hypothesis,
\begin{align*}
\chi(X) 
&=\chi(X_n) + \chi(Y) - \chi(Y \cap X_n)  \\ 
&= \chi(X_n) + \sum_{i<n} \chi(X_i) - \sum_{i<j<n} \chi(X_i \cap X_j) - \sum_{i <n} \chi(X_i \cap X_n)\\
&= \sum_i \chi(X_i) - \sum_{i<j} \chi(X_i \cap X_j).\qedhere
\end{align*}
\end{proof}

\begin{lemma}
\label{lemm:tree}
Assume that $k$ is perfect. Let $X$ be a complete $k$-variety such that $\Gamma_X$ is a tree and every irreducible component of the $\cl{k}$-variety $X_{\cl{k}}$ is isomorphic to $\Pp^1_{\cl{k}}$. If $X$ supports a zero-cycle of odd degree, then $X(k) \neq \varnothing$.
\end{lemma}
\begin{proof}
We first reduce to the case when $X$ is irreducible by induction on the number of irreducible components of $X$. Let us assume that $X$ is not irreducible, and let $T$ be the irreducible component of $X$ corresponding to the $\Gal(\cl{k}/k)$-orbit of a leaf in $\Gamma_X$. Note that the closure $Y$ of $X-T$ in $X$ is non-empty. We may view the graphs $\Gamma_T$ and $\Gamma_Y$ as full subgraphs of $\Gamma_X$, compatibly with the $\Gal(\cl{k}/k)$-actions, and the set of vertices of $\Gamma_X$ is the disjoint union of the set of vertices of $\Gamma_T$ and $\Gamma_Y$. Since each vertex of $\Gamma_T$ is a leaf of the tree $\Gamma_X$, it follows that $\Gamma_Y$ is a tree (removing a leaf from a tree yields a tree or the empty graph). The irreducible components of $Y_{\cl{k}}$, being among those of $X_{\cl{k}}$, are isomorphic to $\Pp^1_{\cl{k}}$.

Let $P$ be an irreducible component of $T_{\cl{k}}$. We claim that in the graph $\Gamma_X$ there is exactly one edge one of whose extremities is $P$, and that the other extremity is an irreducible component of $Y_{\cl{k}}$. Since $P$ is a leaf in the tree $\Gamma_X$, there is at most one such edge. If there are none, then $P$ is the only vertex of $\Gamma_X$, hence $P=X_{\cl{k}}$ and $X$ is irreducible, a contradiction. Thus there is exactly one such edge, let $Q$ be its other extremity. If $Q$ is an irreducible component of $T_{\cl{k}}$, it is a leaf of the tree $\Gamma_X$, and $P,Q$ are the only two vertices of $\Gamma_X$. Then $T=X$ is again irreducible, a contradiction. Therefore $Q$ is an irreducible component of $Y_{\cl{k}}$, which proves the claim.

From the claim we deduce that the irreducible components of $T_{\cl{k}}$ are pairwise disjoint (so that $m_T=\mu_T$), and are in bijection with the points  of the finite $\cl{k}$-variety $(T \cap Y)_{\cl{k}}$. Thus $\mu_{T \cap Y} = \mu_T$. If $n_T$ is even, then $n_Y$ must be odd. If $n_T$ is odd, then so is $\mu_T$  by \autoref{lemm:mu_n}, and thus also $\mu_{T \cap Y}$. The latter is the sum of the separable degrees of the residue fields of the points of the finite variety $T \cap Y$, which are in particular closed points of $Y$. Since $k$ is perfect, we deduce again that $n_Y$ must be odd. Therefore in any case we may conclude the proof using the induction hypothesis for the variety $Y$.

Thus we may assume that $X$ is irreducible. Since the group $\Gal(\cl{k}/k)$ acts transitively on the tree $\Gamma_X$, each of the vertices of $\Gamma_X$ is a leaf. It follows that the graph $\Gamma_X$ has no more than two vertices (and thus at most one edge). If $X_{\cl{k}}$ has two irreducible components, they meet in a single point (over $\cl{k}$), which coincides with $\Sing(X_{\cl{k}}) = \Sing(X)_{\cl{k}}$ \cite[(6.7.4)]{ega-4-2}. Since $k$ is perfect, it follows that $\Sing(X)$ is a rational point of $X$. Otherwise, the variety $X$ is geometrically irreducible. Thus $X_{\cl{k}} \simeq \Pp^1_{\cl{k}}$, hence $X$ is a smooth conic over $k$. Since $X$ supports a zero-cycle of odd degree, it must possess a rational point by Springer's theorem.
\end{proof}

\begin{remark}
It is not necessary to assume that $X$ is the support of a snc divisor in \autoref{prop:h1}. An earlier version of this paper contained a proof of this more general statement, but as pointed out to me by Johannes Nicaise, \autoref{prop:h1} suffices for the proof of \autoref{th:three} and is substantially shorter to prove.
\end{remark}


\begin{thebibliography}{{Pop}14}
\bibitem[CT05]{Colliot-Finitude}
Jean-Louis Colliot-Th{\'e}l{\`e}ne.
\newblock Un th\'eor\`eme de finitude pour le groupe de {C}how des
  z\'ero-cycles d'un groupe alg\'ebrique lin\'eaire sur un corps {$p$}-adique.
\newblock {\em Invent. Math.}, 159(3):589--606, 2005.

\bibitem[CP14]{CP-3}
Vincent Cossart and Olivier Piltant.
\newblock Resolution of singularities of arithmetical threefolds {II}.
\newblock {\em Preprint}, 2014.
\newblock \href{http://arxiv.org/abs/1412.0868}{\texttt{arXiv:1412.0868}}.

\bibitem[DL76]{Deligne-Lusztig}
P.~Deligne and G.~Lusztig.
\newblock Representations of reductive groups over finite fields.
\newblock {\em Ann. of Math. (2)}, 103(1):103--161, 1976.

\bibitem[Edi92]{Edixhoven-Neron}
Bas Edixhoven.
\newblock N\'eron models and tame ramification.
\newblock {\em Compositio Math.}, 81(3):291--306, 1992.

\bibitem[EN11]{Esnault-Nicaise}
H{\'e}l{\`e}ne Esnault and Johannes Nicaise.
\newblock Finite group actions, rational fixed points and weak {N}\'eron
  models.
\newblock {\em Pure Appl. Math. Q.}, 7(4, Special Issue: In memory of Eckart
  Viehweg):1209--1240, 2011.

\bibitem[Ful98]{Ful-In-98}
William Fulton.
\newblock {\em Intersection theory}, volume~2 of {\em Ergebnisse der Mathematik
  und ihrer Grenzgebiete. 3. Folge. A Series of Modern Surveys in Mathematics}.
\newblock Springer-Verlag, Berlin, second edition, 1998.

\bibitem[GD75]{GD-Aut-affine}
M.~H. Gizatullin and V.~I. Danilov.
\newblock Automorphisms of affine surfaces. {I}.
\newblock {\em Izv. Akad. Nauk SSSR Ser. Mat.}, 39(3):523--565, 703, 1975.

\bibitem[GLL13]{Index}
Ofer Gabber, Qing Liu, and Dino Lorenzini.
\newblock The index of an algebraic variety.
\newblock {\em Invent. Math.}, 192(3):567--626, 2013.

\bibitem[Gro65]{ega-4-2}
Alexander Grothendieck.
\newblock \'{E}l\'ements de g\'eom\'etrie alg\'ebrique. {IV}. \'{E}tude locale
  des sch\'emas et des morphismes de sch\'emas. {II}.
\newblock {\em Inst. Hautes \'Etudes Sci. Publ. Math.}, (24):231, 1965.

\bibitem[Gro67]{ega-4-4}
Alexander Grothendieck.
\newblock \'{E}l\'ements de g\'eom\'etrie alg\'ebrique. {IV}. \'{E}tude locale
  des sch\'emas et des morphismes de sch\'emas {IV}.
\newblock {\em Inst. Hautes \'Etudes Sci. Publ. Math.}, (32):361, 1967.

\bibitem[Hir64]{Hir-64}
Heisuke Hironaka.
\newblock Resolution of singularities of an algebraic variety over a field of
  characteristic zero. {I}, {II}.
\newblock {\em Ann. of Math. (2)}, 79:109--203, 205--326, 1964.

\bibitem[ILO]{Gabber-book}
Luc Illusie, Yves Laszlo, and Fabrice Orgogozo.
\newblock {\em Travaux de Gabber sur l'uniformisation locale et la cohomologie
  \'etale des sch\'emas quasi-excellents. S\'eminaire \`a l'Ecole polytechnique
  2006--2008}.
\newblock Avec la collaboration de Fr\'ed\'eric D\'eglise, Alban Moreau,
  Vincent Pilloni, Michel Raynaud, Jo\"el Riou, Beno\^it Stroh et Michael
  Temkin. \href{http://arxiv.org/abs/1207.3648}{\texttt{arXiv:1207.3648}}.

\bibitem[Kam75]{Kam-forms}
T.~Kambayashi.
\newblock On the absence of nontrivial separable forms of the affine plane.
\newblock {\em J. Algebra}, 35:449--456, 1975.

\bibitem[Kra96]{Kraft-Bourbaki}
Hanspeter Kraft.
\newblock Challenging problems on affine {$n$}-space.
\newblock {\em Ast\'erisque}, (237):Exp.\ No.\ 802, 5, 295--317, 1996.
\newblock S{\'e}minaire Bourbaki, Vol. 1994/95.

\bibitem[KS08]{Kato-Saito}
Kazuya Kato and Takeshi Saito.
\newblock Ramification theory for varieties over a perfect field.
\newblock {\em Ann. of Math. (2)}, 168(1):33--96, 2008.

\bibitem[Lan02]{Lang-Algebra}
Serge Lang.
\newblock {\em Algebra}, volume 211 of {\em Graduate Texts in Mathematics}.
\newblock Springer-Verlag, New York, third edition, 2002.

\bibitem[Lip78]{Lipman-desingularization}
Joseph Lipman.
\newblock Desingularization of two-dimensional schemes.
\newblock {\em Ann. Math. (2)}, 107(1):151--207, 1978.

\bibitem[Liu02]{Liu}
Qing Liu.
\newblock {\em Algebraic geometry and arithmetic curves}, volume~6 of {\em
  Oxford Graduate Texts in Mathematics}.
\newblock Oxford University Press, Oxford, 2002.
\newblock Translated from the French by Reinie Ern{\'e}, Oxford Science
  Publications.

\bibitem[L{\"u}t93]{L-Nagata}
W.~L{\"u}tkebohmert.
\newblock On compactification of schemes.
\newblock {\em Manuscripta Math.}, 80(1):95--111, 1993.

\bibitem[Mil80]{Milne-etale}
James~S. Milne.
\newblock {\em \'{E}tale cohomology}, volume~33 of {\em Princeton Mathematical
  Series}.
\newblock Princeton University Press, Princeton, N.J., 1980.

\bibitem[Pop96]{Pop-large}
Florian Pop.
\newblock Embedding problems over large fields.
\newblock {\em Ann. of Math. (2)}, 144(1):1--34, 1996.

\bibitem[{Pop}14]{Pop-Survey}
Florian {Pop}.
\newblock {Little survey on large fields -- old \& new.}
\newblock In {\em {Valuation theory in interaction. Proceedings of the 2nd
  international conference and workshop on valuation theory, Segovia and El
  Escorial, Spain, July 18--29, 2011}}, pages 432--463. Z\"urich: European
  Mathematical Society (EMS), 2014.

\bibitem[Qui73]{Qui-72}
Daniel Quillen.
\newblock Higher algebraic {$K$}-theory. {I}.
\newblock In {\em Algebraic {$K$}-theory, {I}: {H}igher {$K$}-theories ({P}roc.
  {C}onf., {B}attelle {M}emorial {I}nst., {S}eattle, {W}ash., 1972)}, pages
  85--147. Lecture Notes in Math., Vol. 341. Springer, Berlin, 1973.

\bibitem[Ser09]{Serre-finite_fields}
Jean-Pierre Serre.
\newblock How to use finite fields for problems concerning infinite fields.
\newblock In {\em Arithmetic, geometry, cryptography and coding theory}, volume
  487 of {\em Contemp. Math.}, pages 183--193. Amer. Math. Soc., Providence,
  RI, 2009.

\bibitem[Sha66]{Sha-infinite}
I.~R. Shafarevich.
\newblock On some infinite-dimensional groups.
\newblock {\em Rend. Mat. e Appl. (5)}, 25(1-2):208--212, 1966.


\bibitem[SGA72]{SGA-4-2}
{\em Th\'eorie des topos et cohomologie \'etale des sch\'emas. {T}ome 2}.
\newblock Lecture Notes in Mathematics, Vol. 270. Springer-Verlag, Berlin-New
  York, 1972.
\newblock S{\'e}minaire de G{\'e}om{\'e}trie Alg{\'e}brique du Bois-Marie
  1963--1964 (SGA 4), Dirig{\'e} par M. Artin, A. Grothendieck et J. L.
  Verdier. Avec la collaboration de N. Bourbaki, P. Deligne et B. Saint-Donat.


\bibitem[SGA73]{SGA-4-3}
{\em Th\'eorie des topos et cohomologie \'etale des sch\'emas. {T}ome 3}.
\newblock Lecture Notes in Mathematics, Vol. 305. Springer-Verlag, Berlin-New
  York, 1973.
\newblock S{\'e}minaire de G{\'e}om{\'e}trie Alg{\'e}brique du Bois-Marie
  1963--1964 (SGA 4), Dirig{\'e} par M. Artin, A. Grothendieck et J. L.
  Verdier. Avec la collaboration de P. Deligne et B. Saint-Donat.


\bibitem[{Sta}15]{stacks}
The {Stacks Project Authors}.
\newblock {S}tacks {P}roject.
\newblock \url{http://stacks.math.columbia.edu}, 2015.

\end{thebibliography}
\end{document}